\documentclass[arxiv]{agn_article}
\PassOptionsToPackage{sortcites}{agn_bib} 
\usepackage[french,english]{babel}
\usepackage{csquotes}

\usepackage{amssymb,amsmath,amsthm,agn_bib}

\usepackage{agn_macros,agn_authors}
\usepackage{mathrsfs}

\AtBeginBibliography{\footnotesize}

\newtheorem{theorem}{Theorem}[section]

\newtheorem{lemma}[theorem]{Lemma}
\newtheorem{observation}[theorem]{Observation}
\newtheorem{corollary}[theorem]{Corollary}

\theoremstyle{definition}

\newtheorem{remark}[theorem]{Remark}

\usepackage{xcolor,hyperref}
\definecolor{darkblue}{rgb}{0,0,0.6}
\hypersetup{
   colorlinks=true,       
   linkcolor=darkblue,          
   citecolor=darkblue,        
   filecolor=darkblue,      
   urlcolor=darkblue           
}

\usepackage[colorinlistoftodos,prependcaption,textsize=tiny]{todonotes}
\newcommand{\skalarProd}[2]{\big\langle#1,#2\big\rangle}

\newcommand{\ntimes}{\underline{\times}}
\newcommand{\nCurl}{\underline{\Curl}\,}
\newcommand{\ncurl}{\underline{\operatorname{curl}}\,}
\renewcommand{\D}{\operatorname{D}\hspace{-1pt}}

\usepackage{tikz}

\begin{document}
\begin{tikzpicture}[remember picture, overlay]
 \node [xshift=-1cm,yshift=15cm,rotate=-90] at (current page.south east)
 {Comptes Rendus. Mathématique \textbf{359}, issue 6 (2021), doi: \href{https://doi.org/10.5802/crmath.216}{10.5802/crmath.216}.
 };
\end{tikzpicture}
\numberwithin{equation}{section}

 \title{$L^p$-versions of generalized Korn inequalities for incompatible tensor fields in arbitrary dimensions with $p$-integrable exterior derivative}
\knownauthors[lewintan]{lewintan,neff}

\subtitle{\LARGE Versions $L^p$ des in\'{e}galit\'{e}s g\'{e}n\'{e}ralis\'{e}es de Korn pour les champs de tenseurs incompatibles de dimension quelconque avec d\'{e}riv\'{e}e ext\'{e}rieure $p$-int\'{e}grable}

\maketitle

\begin{abstract}
For $n\ge2$ and $1<p<\infty$ we prove an $L^p$-version of the generalized Korn-type inequality for
incompatible, $p$-integrable tensor fields $P:\Omega \to \R^{n\times n}$ having $p$-integrable generalized $\nCurl$ and generalized vanishing tangential trace $P\,\tau_l=0$  on $\partial \Omega$, denoting by $\{\tau_l\}_{l=1,\ldots, n-1}$ a moving tangent frame on $\partial\Omega$, more precisely we have:
   \begin{equation*}
    \norm{ P }_{L^p(\Omega,\R^{n\times n})}\leq c\,\left(\norm{ \sym P }_{L^p(\Omega,\R^{n \times n})}+ \norm{ \nCurl P }_{L^p(\Omega,(\so(n))^n)}\right),
   \end{equation*}
  where the generalized $\nCurl$ is given by $(\nCurl P)_{ijk} \coloneqq \partial_i P_{kj}-\partial_j  P_{ki}$ and $c=c(n,p,\Omega)>0$.
\end{abstract}

\begin{otherlanguage}{french}
\begin{abstract}
On montre pour $n\ge2$ et $1<p<\infty$ une version $L^p$ de l'in\'{e}galit\'{e} g\'{e}n\'{e}ralis\'{e}e de Korn pour tous les champs de tenseurs incompatibles et $p$-int\'{e}grable $P:\Omega \to \R^{n\times n}$ avec rotationnel g\'{e}n\'{e}ralis\'{e}  $p$-int\'{e}grable et avec z\'{e}ro trace tangentielle 
 $P\,\tau_l=0$  sur $\partial\Omega$ o\`{u} $\{\tau_l\}_{l=1,\ldots, n-1}$ est un rep\`{e}re tangent sur $\partial\Omega$. Plus pr\'{e}cis\'{e}ment on a
  \begin{equation*}
  \norm{ P }_{L^p(\Omega,\R^{n\times n})}\leq c\,\left(\norm{ \sym P }_{L^p(\Omega,\R^{n \times n})}+ \norm{ \nCurl P }_{L^p(\Omega,(\so(n))^n)}\right),
 \end{equation*}
 o\`{u} les composantes  du rotationnel g\'{e}n\'{e}ralis\'{e} s'\'{e}crivent  $(\nCurl P)_{ijk} \coloneqq \partial_i P_{kj}-\partial_j  P_{ki}$ et $c=c(n,p,\Omega)>0$.
\end{abstract}
\end{otherlanguage}



\section{Introduction}
On montre pour $n\ge2$ et $1<p<\infty$ une version $L^p$ de l'in\'{e}galit\'{e} g\'{e}n\'{e}ralis\'{e}e de Korn pour tous les champs de tenseurs incompatibles et $p$-int\'{e}grable $P:\Omega \to \R^{n\times n}$ avec rotationnel g\'{e}n\'{e}ralis\'{e}  $p$-int\'{e}grable et avec z\'{e}ro trace tangentielle 
 $P\,\tau_l=0$  sur $\partial\Omega$ o\`{u} $\{\tau_l\}_{l=1,\ldots, n-1}$ est un rep\`{e}re tangent sur $\partial\Omega$. Plus pr\'{e}cis\'{e}ment on a
  \begin{equation*}
  \norm{ P }_{L^p(\Omega,\R^{n\times n})}\leq c\,\left(\norm{ \sym P }_{L^p(\Omega,\R^{n \times n})}+ \norm{ \nCurl P }_{L^p(\Omega,(\so(n))^n)}\right),
 \end{equation*}
 o\`{u} les composantes  du rotationnel g\'{e}n\'{e}ralis\'{e} s'\'{e}crivent  $(\nCurl P)_{ijk} \coloneqq \partial_i P_{kj}-\partial_j  P_{ki}$ et $c=c(n,p,\Omega)>0$.

 \section{Introduction}
In \cite{agn_lewintan2019KornLp} we have shown that there exists a constant $c=c(p,\Omega)>0$ such that
\begin{equation*}
 \norm{ P }_{L^p(\Omega,\R^{3\times3})}\leq c\,\left(\norm{ \sym P }_{L^p(\Omega,\R^{3\times3})} + \norm{ \Curl P }_{L^p(\Omega,\R^{3\times3})}\right)
\end{equation*}
holds for all tensor fields $P\in  W^{1,\,p}_0(\Curl; \Omega,\R^{3\times3})$, i.e., for all $P\in W^{1,\,p}(\Curl; \Omega,\R^{3\times3})$ with vanishing tangential trace $ P\times \nu=0 $ ( $\Leftrightarrow~  P\,\tau_l=0$ ) on $ \partial\Omega$ 
where  $\nu$ denotes the outward unit normal vector field and $\{\tau_l\}_{l=1,2,3}$  a moving tangent frame on $\partial\Omega$ and  $\Omega\subset\R^3$ is a bounded Lipschitz domain. The crucial ingredients for our proof were the Lions lemma and Ne\v{c}as estimate, the compactness of $W^{1,\,p}_0(\Omega)\subset\!\subset L^p(\Omega)$ and an algebraic identity in terms of components of the cross product of a skew-symmetric matrix with a vector. Recall, that for a bounded Lipschitz domain (i.e.~bounded open connected with Lipschitz boundary) $\Omega\subset\R^n$, the Lions lemma states that $f\in L^p(\Omega)$ if and only if $f\in W^{-1,\,p}(\Omega)$ and $\nabla f \in W^{-1,\,p}(\Omega,\R^n)$, which is equivalently expressed by the Ne\v{c}as estimate
\begin{equation}
 \norm{f}_{L^p(\Omega)}\le c\, (\norm{f}_{ W^{-1,\,p}(\Omega)}+\norm{\nabla f}_{ W^{-1,\,p}(\Omega,\R^n)})
\end{equation}
with a positive constant $c=c(p,n,\Omega)$. In fact, such an argumentation scheme is also used to prove the classical Korn inequalities, cf. e.g. \cite{Ciarlet2010, Ciarlet2013FAbook, Ciarlet2005korn, CMM2018,Geymonat86,agn_lewintan2019KornLp} and the discussions contained therein. However, \cite{Ciarlet2010, Ciarlet2013FAbook, Ciarlet2005korn, CMM2018,Geymonat86} focus on the compatible case, i.e.~$P=\D u$, where we deal with general square matrices $P\in\R^{n\times n}$, thus, the incompatible case.

Here, we extend our results from \cite{agn_lewintan2019KornLp} to the $n$-dimensional case, hence generalizing the main result from \cite{agn_neff2012maxwell} to the $L^p$-setting. This is, we prove 
 \begin{equation}
  \norm{ P }_{L^p(\Omega,\R^{n\times n})}\leq c\,\left(\norm{ \sym P }_{L^p(\Omega,\R^{n \times n})}+ \norm{ \nCurl P }_{L^p(\Omega,(\so(n))^n)}\right) \quad \forall \ P \in W^{1,\,p}_0(\nCurl; \Omega,\R^{n\times n}),
 \end{equation}
where  the generalized $\nCurl$ is given by $(\nCurl P)_{ijk} \coloneqq \partial_i P_{kj}-\partial_j  P_{ki}$ and the vanishing tangential trace condition reads \ $ P\,\tau_l=0 $ \ on $ \partial\Omega$
denoting by  $\{\tau_l\}_{l=1,\ldots, n-1}$ a moving tangent frame on $\partial\Omega$.
 
For a detailed motivation and definitions we refer to \cite{agn_lewintan2019KornLp} and the references contained therein.  Indeed, we follow the argumentation scheme presented in \cite{agn_lewintan2019KornLp} closely, emphasizing only the necessary modifications coming from the generalization of the vector product. The latter then provides an adequate generalization of the $\Curl$-operator to the $n$-dimensional setting. Especially, the generalized $\operatorname{curl}$ of vector fields can be seen as their exterior derivative, see also the discussion in \cite{agn_neff2012maxwell}.

\section{Notations}
Let $n\ge2$. For vectors $a,b\in\R^n$, we consider the scalar product  $\skalarProd{a}{b}\coloneqq\sum_{i=1}^n a_i\,b_i \in \R$, the (squared) norm  $\norm{a}^2\coloneqq\skalarProd{a}{a}$ and  the dyadic product  $a\otimes b \coloneqq \left(a_i\,b_j\right)_{i,j=1,\ldots,n}\in \R^{n\times n}$. Similarly, for matrices $P,Q\in\R^{n\times n}$ we define the scalar product $\skalarProd{P}{Q} \coloneqq\sum_{i,j=1}^n P_{ij}\,Q_{ij} \in \R$ and the (squared) Frobenius-norm $\norm{P}^2\coloneqq\skalarProd{P}{P}$.
Moreover, $P^T\coloneqq (P_{ji})_{i,j=1,\ldots,n}$ denotes the transposition of the matrix $P=(P_{ij})_{i,j=1,\ldots,n}$, which decomposes orthogonally into the symmetric part $\sym P \coloneqq \frac12\left(P+P^T\right)$ and the skew-symmetric part $\skew P \coloneqq \frac12\left(P-P^T\right)$.
The Lie-Algebra of skew-symmetric matrices is denoted by $\so(n)\coloneqq \{A\in\R^{n\times n}\mid A^T = -A\}$. The identity matrix is denoted by $\id$, so that the trace of a matrix $P$ is given by \ $\tr P \coloneqq \skalarProd{P}{\id}$.

\noindent
The cross product for vectors $a,b\in\R^n$ generalizes to
\begin{equation}
 a\ntimes b  \coloneqq \big(a_i\,b_j - a_j\, b_i\big)_{i,j=1,\ldots,n} = a\otimes b - b\otimes a= 2\cdot\skew(a\otimes b) \in \so(n)\cong \R^{\frac{n(n-1)}{2}}.
\end{equation}
Using the bijection $\axl:\so(3)\to\R^3$ we obtain back the standard cross product for $a,b\in\R^3$:
\begin{equation}
 a\times b = -\axl(a\ntimes b)
\end{equation}
where $\axl:\so(3)\to\R^3$ is given in such a way that
\begin{equation}
 A\,b =\axl(A)\times b \quad \forall\, A\in\so(3),\quad b\in\R^3.
\end{equation}

Like in $3$-dimensions it holds:

\begin{observation}\label{obs:parallel}
 Let $n\ge2$. For non-zero vectors $a,b\in\R^n$ we have $a\ntimes b=0$ if and only if $a$ and $b$ are parallel.
\end{observation}
\begin{proof}
 Since the "if" part is obvious we show the "only if" direction:
 \begin{align*}
 a\ntimes b = 0 \quad &\Leftrightarrow \quad \skew (a\otimes b) = 0 \quad \Leftrightarrow \quad 
  a\otimes b = b \otimes a \quad \Rightarrow \quad (a\otimes b) b = (b\otimes a) b \\
  & \Leftrightarrow \quad a\,\norm{b}^2= b\,\skalarProd{a}{b}. \qedhere
 \end{align*}
\end{proof}

As in the $3$-dimensional case, we understand the vector product of a square-matrix $P\in\R^{n\times n}$  and a vector $b\in \R^n$ row-wise, i.e.
\begin{equation}
 P\ntimes b \coloneqq \big((P^T e_k) \ntimes b\big)_{k=1,\ldots,n} = \big(P_{ki}\,b_j-P_{kj}\,b_i\big)_{i,j,k=1,\ldots n} \in(\so(n))^n.
\end{equation}
For index notations we set: \quad $(P\ntimes b)_{ijk}\coloneqq P_{ki}\,b_j-P_{kj}\,b_i $.

Especially, for skew-symmetric matrices $A\in\so(n)$ we note the following crucial relation for our considerations:
\begin{equation}\label{eq:crucial}
\begin{split}
 (A\ntimes b)_{kij} - (A\ntimes b)_{kji} + (A\ntimes b)_{jik} &= A_{jk}\,b_i-A_{ji}\,b_k-(A_{ik}\,b_j-A_{ij}\,b_k)+A_{kj}\,b_i-A_{ki}\,b_j \\
 &\overset{\mathclap{(A_{ij}=-A_{ji})}}{=} \qquad 2 A_{ij}\,b_k \qquad \forall\, i,j,k=1,\ldots n
\end{split}
\end{equation}
with the direct consequence
\begin{observation}\label{obs:alg_statement}
 Let $n\ge2$. For $A\in\so(n)$ and a non-zero vector $b\in\R^n$ we have $A\ntimes b = 0$ if and only if $A=0$.
\end{observation}

\noindent
Let $\Omega\subset \R^n$, $n\ge2$, be a domain. As in $\R^3$ we formally introduce the generalized $\ncurl$ of a vector field $v\in\mathscr{D}'(\Omega,\R^n)$
via
\begin{equation}
 \ncurl\, v \coloneqq v\ntimes (-\nabla) = \nabla \ntimes v = -2 \cdot \skew (v\otimes \nabla) = -2 \cdot \skew (\D v) \in \so(n).
\end{equation}
Furthermore, for $(n\times n)$-square matrix fields we understand this operation row-wise:
\begin{equation}
 \nCurl P \coloneqq P \ntimes (-\nabla) = \big(\ncurl (P^Te_k)\big)_{k=1,\ldots,n} = \big(\partial_i P_{kj}-\partial_j  P_{ki}\big)_{i,j,k=1,\ldots,n} \in (\so(n))^n .
\end{equation}
For index notations we define: \quad $(\nCurl P)_{ijk} \coloneqq \partial_i P_{kj}-\partial_j  P_{ki}$.~  Of course, \quad $ \nCurl \D v \equiv 0$.\\
Moreover, we make use of the generalized divergence $\Div$ for matrix fields $P\in\mathscr{D}'(\Omega,\R^{n\times n})$ row-wise, via
\begin{equation}
 \Div P \coloneqq \big(\div (P^Te_k)\big)_{k=1,\ldots,n}\,.
\end{equation}

In fact, the crucial relation \eqref{eq:crucial} implies that the full gradient of a skew-symmetric matrix is already determined by its generalized $\nCurl$, cf. also \cite[p.\,155]{agn_munch2008curl}:
\begin{corollary}\label{cor:lin_combi}
Let $n\ge2$.  For $A\in\mathscr{D}'(\Omega,\so(n))$ the entries of the gradient $\D A$ are linear combinations of the entries from $\nCurl A$.
\end{corollary}
\begin{proof}
 Replacing $b$ by $-\nabla$ in \eqref{eq:crucial} we see that
 \begin{equation*}
  (\nCurl A)_{kij} - (\nCurl A)_{kji} + (\nCurl A)_{jik}  = -2\,\partial _k A_{ij}.\qedhere
 \end{equation*}
\end{proof}
\noindent
This control of all first partial derivatives of a skew-symmetric matrix field in terms of the generalized $\nCurl$ then immediately yields in all dimensions 
\begin{corollary}\label{cor:Nye}
Let $n\ge2$.  For $A\in L^p(\Omega,\so(n))$ we have $\nCurl A \equiv 0$ in the distributional sense if and only if $A=\operatorname{const}$ almost everywhere in $\Omega$.
\end{corollary}

\subsection{Function spaces}

Having above relations at hand we can now catch up the arguments from \cite{agn_lewintan2019KornLp}. For that purpose let us define for $n\ge2$ and $1<p<\infty$ the space
\begin{subequations}
\begin{align}
  W^{1,\,p}(\nCurl; \Omega,\R^{n\times n}) &\coloneqq \{P\in L^p(\Omega,\R^{n\times n})\mid \nCurl P \in L^p(\Omega,(\so(n))^n)\}
 \shortintertext{equipped with the norm}
 \norm{P}_{ W^{1,\,p}(\nCurl; \Omega,\R^{n\times n})}&\coloneqq \left(\norm{P}^p_{L^p(\Omega,\R^{n\times n})} + \norm{\nCurl P}^p_{L^p(\Omega,(\so(n))^n)} \right)^{\frac{1}{p}}.
\end{align}
\end{subequations}
By definition of the norm in the dual space, we have
\begin{equation}\label{eq:Curlemddual}
\begin{split}
 P\in L^p(\Omega,\R^{n\times n}) \quad \Rightarrow \quad &\nCurl P\in W^{-1,\,p}(\Omega,(\so(n))^n)\\
 \text{ with }  & \norm{\nCurl P}_{W^{-1,\,p}(\Omega,(\so(n))^n)} \leq c\, \norm{P}_{L^p(\Omega,\R^{n\times n})}.
\end{split}
\end{equation}
\noindent
Furthermore, we consider the subspace
\begin{equation}
 \begin{split}
  W^{1,\,p}_0(\nCurl; \Omega,\R^{n\times n}) \coloneqq\, &\{P\in  W^{1,\,p}(\nCurl; \Omega,\R^{n\times n}) \mid P \ntimes \nu = 0 \text{ on } \partial \Omega\}\\
   =\,& \{P\in  W^{1,\,p}(\nCurl; \Omega,\R^{n\times n}) \mid P\,\tau_l = 0 \text{ on } \partial \Omega \text{ for all $l=1,\ldots,n-1$}\},
 \end{split}
\end{equation}
where $\nu$ stands for the outward unit normal vector field and $\{\tau_l\}_{l=1,\ldots, n-1}$ denotes a moving tangent frame  on $\partial\Omega$. Here, the generalized tangential trace $P\ntimes \nu$ is understood in the sense of $W^{-\frac1p,\, p}(\partial \Omega,\R^{n\times n})$ which is justified by partial integration, so that its trace is defined by
\begin{align}
 \forall\ k=1,\ldots n,\  \forall\ Q\in  W^{1-\frac{1}{p'},\,p'}(\partial\Omega,\R^{n\times n}):& \\ \skalarProd{(P^Te_k)\ntimes \nu}{Q}_{\partial \Omega} &=  \int_{\Omega}\skalarProd{\ncurl (P^Te_k)}{\widetilde{Q}}_{\R^{n\times n}}+2\skalarProd{P^Te_k}{\Div (\skew \widetilde{Q})}_{\R^n}\, \intd{x} \notag
\end{align}
having denoted by $\widetilde{Q}\in W^{1,\,p'}(\Omega,\R^{n\times n})$ any extension of $Q$ in $\Omega$, where, $\skalarProd{.}{.}_{\partial\Omega}$ indicates the duality pairing between $W^{-\frac1p,\,p}(\partial\Omega,\R^{n\times n})$ and $W^{1-\frac{1}{p'},\,p'}(\partial\Omega,\R^{n\times n})$. Indeed, for $P,Q\in C^1(\Omega,\R^{n\times n})\cap C^0(\overline{\Omega},\R^{n\times n})$ we have
\begin{equation}\label{eq:prelim}
\begin{split}
 \frac12\skalarProd{(P^Te_k)\ntimes \nu}{Q}_{\R^{n\times n}}& = \skalarProd{\skew ((P^Te_k)\otimes \nu)}{Q}_{\R^{n\times n}}=\skalarProd{(P^Te_k)\otimes \nu}{\skew Q}_{\R^{n\times n}} \\&= \sum_{i,j=1}^n  P_{ki}\,\nu_j(\skew Q)_{ij} = -\sum_{i,j=1}^n \nu_j\, (\skew Q)_{ji}\,P_{ki} \\
  &=- \skalarProd{\nu}{(\skew Q)\, (P^Te_k)}_{\R^n},
\end{split}
\end{equation}
so that using the divergence-theorem, for $k=1,\ldots,n$ we have\footnote{This partial integration formula slightly differs from the situation in $\R^3$ since the generalized $\nCurl$ has image in $(\so(n))^n$ which corresponds to $\R^{n\times n}$ only for $n=3$.} 
\begin{align}
 \int_{\partial \Omega} \skalarProd{(P^Te_k)\ntimes \nu}{Q}_{\R^{n\times n}} \,\intd{S} & \overset{\eqref{eq:prelim}}{=} -2 \int_{\partial \Omega} \skalarProd{\nu}{(\skew Q)\, (P^Te_k)}_{\R^n} \, \intd{S} = -2 \int_{\Omega}\div ((\skew Q)\, (P^Te_k))\,\intd{x}\notag \\
 & = -2 \int_{\Omega} \skalarProd{\Div [(\skew Q)^T]}{P^Te_k}_{\R^n}+ \skalarProd{(\skew Q)}{\D (P^T e_k)}_{\R^{n\times n}}\,\intd{x}\\
 & = \int_{\Omega}\skalarProd{\ncurl (P^Te_k)}{Q}_{\R^{n\times n}}+2\skalarProd{P^Te_k}{\Div (\skew Q)}_{\R^n}\, \intd{x}.\notag
\end{align}

Further, following \cite{agn_lewintan2019KornLp} we introduce also the space $W^{1,\,p}_{\Gamma,0}(\Curl;\Omega,\R^{n\times n})$  of functions with vanishing tangential trace only on a relatively open (non-empty) subset $\Gamma \subseteq\partial \Omega$ of the boundary by completion of  $C^\infty_{\Gamma,0}(\Omega,\R^{n\times n})$ with respect to the $W^{1,\,p}(\Curl;\Omega,\R^{n\times n})$-norm.

\begin{remark}[Tangential trace condition]\label{rem:trace}
 Note, that the vanishing of the tangential trace $P\ntimes \nu$ at some point is equivalent to $P\,\tau_l=0$ for all $l=1,\ldots,n-1$, denoting by $\{\tau_l\}_{l=1,\ldots, n-1}$ a  frame  of the corresponding tangent space. Indeed, by Observation \ref{obs:parallel}  we have 
 \begin{equation*}
  \begin{split}
     P\ntimes \nu = 0 \quad &\Leftrightarrow \quad \skew((P^Te_k)\otimes \nu)=0, \ k =1,\ldots, n, \quad \Leftrightarrow \quad
     (P^Te_k) \text{ parallel to $\nu$ for all $k=1,\ldots,n$}\\
     &\Leftrightarrow \quad   \skalarProd{P^Te_k}{\tau_l}=0 \quad \forall \ l=1,\ldots,n-1, \ \forall \ k=1,\ldots,n \quad \Leftrightarrow   \quad   P\,\tau_l=0 \quad \forall \ l=1,\ldots,n-1\,.
  \end{split}
 \end{equation*}
\end{remark}

\section{Main results}
We will now state the results from \cite{agn_lewintan2019KornLp} in the $n$-dimensional case, for details of the proofs we refer to the corresponding results therein:

\begin{lemma}
   Let $n\ge2$, $\Omega \subset \R^n$ be a bounded Lipschitz domain and $1<p<\infty$. Then $P\in\mathscr{D}'(\Omega,\R^{n\times n})$, $\sym P\in L^p(\Omega,\R^{n\times n})$ and $\nCurl P \in W^{-1,\,p}(\Omega,(\so(n))^n)$ imply $P\in L^p(\Omega,\R^{n\times n})$. Moreover, we have the estimate
  \begin{equation}\label{eq:basic}
   \norm{P}_{L^p(\Omega,\R^{n\times n})} \leq c\, \left(\norm{\skew P}_{W^{-1,\,p}(\Omega,\R^{n\times n})}+\norm{\sym P}_{L^p(\Omega,\R^{n\times n})}+ \norm{ \nCurl P }_{W^{-1,\,p}(\Omega,(\so(n))^n)}\right),
  \end{equation}
  with a constant $c=c(n,p,\Omega)>0$.
\end{lemma}
\begin{proof}
 Use Corollary \ref{cor:lin_combi} and apply the Lions lemma and Ne\v{c}as estimate, \cite[Theorem 2.6]{agn_lewintan2019KornLp} to $\skew P$, cf. proof of \cite[Lemma 3.1]{agn_lewintan2019KornLp}.
\end{proof}

\noindent
The general Korn-type inequalities then follow by eliminating the first term on the right-hand side of \eqref{eq:basic}:

\begin{theorem}\label{thm:main1}
 Let $n\ge2$,  $\Omega \subset \R^n$ be a bounded Lipschitz domain and $1<p<\infty$. There exists a constant $c=c(n,p,\Omega)>0$, such that for all $P\in  L^p(\Omega,\R^{n\times n})$ we have
 \begin{equation}\label{eq:Korn_Lp_w}
   \inf_{A\in\so(n)}\norm{P-A}_{L^p(\Omega,\R^{n\times n})}\leq c\,\left(\norm{ \sym P }_{L^p(\Omega,\R^{n\times n})}+ \norm{ \nCurl P }_{W^{-1,\,p}(\Omega,(\so(n))^n)}\right).
 \end{equation}
\end{theorem}
\begin{proof}
 By Corollary \ref{cor:Nye} the kernel of the right-hand side consists only of constant skew-symmetric matrices:
 \begin{align}\label{eq:kernel}
   K&\coloneqq\{ P\in  L^p( \Omega,\R^{n\times n}) \mid \sym P = 0 \text{ a.e.~and } \nCurl P = 0 \text{ in the distributional sense}\} \notag\\
   &= \{P= A \text{ a.e.} \mid A\in\so(n)\}.
 \end{align}
Then there exist $M\coloneqq \dim K = \frac{n(n-1)}{2}$ linear forms $\ell_{\alpha}$ on $ L^p(\Omega,\R^{n\times n})$ such that $P\in K$ is equal to $0$ if and only if $\ell_\alpha(P)=0$ for all $\alpha=1,\ldots,M$. Exploiting the compactness $L^p(\Omega,\R^{n\times n})\subset\!\subset  W^{-1,\,p}(\Omega,\R^{n\times n})$ allows us to eliminate the first term on the right-hand side of \eqref{eq:basic} so that we arrive at
\begin{equation}\label{eq:hilfsungl}
 \norm{P}_{L^p(\Omega,\R^{n\times n})}\leq c\,\left(\norm{ \sym P }_{L^p(\Omega,\R^{n\times n})}+ \norm{ \nCurl P }_{W^{-1,\,p}(\Omega,(\so(n))^n)}+\sum_{\alpha=1}^M\abs{\ell_\alpha(P)} \right).
\end{equation}
Considering $P-A_P$ in \eqref{eq:hilfsungl}, where the skew-symmetric matrix $A_P\in K$ is chosen in such a way that $\ell_\alpha(P-A_P)=0$ for all $\alpha=1,\ldots,M$, then yields the conclusion, cf. proof of \cite[Theorem 3.4]{agn_lewintan2019KornLp}.
\end{proof}

\noindent
Moreover, the kernel is killed by the tangential trace condition $P\ntimes \nu \equiv 0$ (or $P\,\tau_l\equiv 0$ for all $l=1,\ldots,n-1$), cf. \eqref{eq:kernel} together with Observation \ref{obs:alg_statement} (and also Remark \ref{rem:trace}), so that we arrive at

\begin{theorem}\label{thm:main2}
Let $n\ge2$,  $\Omega \subset \R^n$ be a bounded Lipschitz domain and $1<p<\infty$. There exists a constant $c=c(n,p,\Omega)>0$, such that for all $P\in  W^{1,\,p}_0(\nCurl; \Omega,\R^{n\times n})$ we have
 \begin{equation}\label{eq:Korn_Lp_thm}
     \norm{ P }_{L^p(\Omega,\R^{n\times n})}\leq c\,\left(\norm{ \sym P }_{L^p(\Omega,\R^{n \times n})}+ \norm{ \nCurl P }_{L^p(\Omega,(\so(n))^n)}\right).
 \end{equation}
\end{theorem}
\begin{proof}
 Having Observation \ref{obs:alg_statement} we can closely follow the proof of \cite[Theorem 3.5]{agn_lewintan2019KornLp}.
\end{proof}

\noindent
Similar argumentations show that \eqref{eq:Korn_Lp_thm} also holds true for functions with vanishing tangential trace only on a relatively open (non-empty) subset $\Gamma\subseteq\partial\Omega$ of the boundary, namely

\begin{theorem}\label{thm:main3}
 Let $n\ge2$,  $\Omega \subset \R^n$ be a bounded Lipschitz domain and $1<p<\infty$. There exists a constant $c=c(n,p,\Omega)>0$, such that for all $P\in  W^{1,\,p}_{\Gamma,0}(\nCurl; \Omega,\R^{n\times n})$ we have
 \begin{equation}\label{eq:Korn_Lp_thm3}
     \norm{ P }_{L^p(\Omega,\R^{n\times n})}\leq c\,\left(\norm{ \sym P }_{L^p(\Omega,\R^{n\times n})}+ \norm{ \nCurl P }_{L^p(\Omega,(\so(n))^n)}\right).
 \end{equation}
\end{theorem}

\noindent
Furthermore, Theorem \ref{thm:main3} reduces for compatible $P=\D u$ to a tangential Korn inequality (Corollary \ref{cor:tKorn}) and for skew-symmetric $P=A$ to a Poincar\'{e} inequality in arbitrary dimensions (Corollary \ref{cor:Poin}):
\begin{corollary}\label{cor:tKorn}
  Let $n\ge2$,  $\Omega \subset \R^n$ be a bounded Lipschitz domain and $1<p<\infty$. There exists a constant $c=c(n,p,\Omega)>0$, such that for all $u\in  W^{1,\,p}_{\Gamma,0}(\Omega,\R^{n})$ we have
   \begin{equation}
   \norm{\D u }_{L^p(\Omega,\R^{n\times n})} \le c\, \norm{\sym \D u}_{L^p(\Omega,\R^n)} \quad \text{with } \D u \ntimes \nu = 0 \quad \text{on $\Gamma$.}
 \end{equation}
\end{corollary}

\begin{remark}
  On $\Gamma$  the boundary condition $\D u \ntimes \nu = 0$  is equivalent to $\D u \,\tau_l=0$ for all $l=1,\ldots,n-1$ and is, e.g., fulfilled if $u_{|\Gamma}\equiv \operatorname{const}.$, see Remark \ref{rem:trace}.
\end{remark}

\begin{corollary}\label{cor:Poin}
  Let $n\ge2$, $\Omega \subset \R^n$ be a bounded Lipschitz domain and $1<p<\infty$. There exists a constant $c=c(n,p,\Omega)>0$, such that for all $A\in  W^{1,\,p}_{\Gamma,0}(\nCurl;\Omega,\so(n))=W^{1,\,p}_{\Gamma,0}(\Omega,\so(n))$ we have
 \begin{equation}
\norm{ A}_{L^p(\Omega,\so(n))}\le c\, \norm{\nCurl A}_{L^p(\Omega,(\so(n))^n)} \quad \text{with }  A \ntimes \nu = 0 \ \overset{*}{\Leftrightarrow} \ A=0 \quad \text{on $\Gamma$.}
 \end{equation}
\end{corollary}

\begin{remark}
 The equivalence of condition $\ast$ is seen in the following way: In any  dimension the rank of the skew-symmetric matrix $A$ is an even number, cf. \cite[p. 30]{vdW-Alg2}, and by Remark \ref{rem:trace} the rows $A^Te_k$ are all parallel ($\Leftrightarrow ~ A\,\tau_l=0$ for all $l=1,\ldots,n-1$) such that the rank of $A$ is not greater then $1$.
\end{remark}
 
\subsection*{Acknowledgement}
This work was initiated in the framework of the Priority Programme SPP 2256 'Variational Methods for Predicting Complex Phenomena in Engineering Structures and Materials' funded by the Deutsche Forschungsgemeinschaft (DFG, German research foundation), Project-ID 422730790. The second author was supported within the project 'A variational scale-dependent transition scheme - from Cauchy elasticity to the relaxed micromorphic continuum’ (Project-ID 440935806). Moreover, both authors were supported in the Project-ID 415894848 by the Deutsche Forschungsgemeinschaft.

\printbibliography

\end{document}